\documentclass[12pt,leqno,twoside]{amsart}

\usepackage[latin1]{inputenc}
\usepackage[T1]{fontenc}
\usepackage[colorlinks=true, pdfstartview=FitV, linkcolor=blue, citecolor=blue, urlcolor=blue]{hyperref}
\usepackage{amstext,amsmath,amscd, bezier,indentfirst,amsthm,amsgen,enumerate, geometry}
\usepackage[all,knot,arc,import,poly]{xy}
\usepackage{amsfonts,color}
\usepackage{amssymb}
\usepackage{latexsym}
\usepackage{epsfig}
\usepackage{graphicx}
\usepackage{srcltx}
\usepackage{enumitem}

\newcommand{\fin}{\hspace*{\fill}$\square$\vspace*{2mm}}

\theoremstyle{plain}
\newtheorem{theorem}{Theorem}[section]
\newtheorem{lemma}[theorem]{Lemma}
\newtheorem{proposition}[theorem]{Proposition}

\theoremstyle{definition}
\newtheorem{definition}[theorem]{Definition}

\theoremstyle{remark}
\newtheorem{remark}[theorem]{\sc Remark}
\newtheorem{example}[theorem]{\sc Example}

\def\bN{{\mathbb N}}
\def\bP{{\mathbb P}}

\def\bR{{\mathbb R}}

\def\bX{{\mathbb X}}

\def\cM{{\mathcal M}}
\def\cS{{\mathcal S}}
\def\Si{{ \mathcal{S}_{\infty}}}

\def\Bi{{ \mathcal{B}_{\infty}}}

\def\cK{{\mathcal K}}

\def\ity{\infty}
\def\Arc{{\rm Arc}}
\def\rD{{\rm D}}

\def\ord{{\rm ord}}

\def\graph{{\rm graph}}

\def\Sing{{\rm Sing}}

\def\const.{{\rm const.}}

\def\d{{\rm d}}

\begin{document}
\title[Bifurcation locus of real polynomial maps]{Towards effective detection of the  bifurcation locus of real polynomial maps}

\author{Luis Renato G. Dias}
\address{Faculdade de Matem\'atica, Universidade Federal de Uberl\^andia, Av. Jo\~ao Naves de \'Avila 2121, 1F-153 - CEP: 38408-100, Uberl\^andia, Brazil.}
\email{lrgdias@famat.ufu.br}

\author{Susumu Tanab\'e}
\address{Math\'ematiques,  Universit\'e de Galatasaray, 34357 Istanbul, Turquie.}
\email{tanabesusumu@hotmail.com}

\author{Mihai Tib\u ar}
\address{Univ. Lille, CNRS, UMR 8524 - Laboratoire Paul Painlev\' e, F-59000 Lille, France}
\email{tibar@math.univ-lille1.fr}

\commby{Communicated by Teresa Krick}

\subjclass[2010]{14D06, 14Q20, 58K05, 57R45, 14P10, 32S20, 58K15}

\keywords{bifurcation locus, real polynomial maps, regularity at infinity, detection}

\thanks{LRGD and MT acknowledge support from the USP-COFECUB Uc
Ma 133/12 grant. LRGD acknowledges  support from  the Fapemig-Proc  APQ-00431-14 grant.
 ST and MT acknowledge support from the CNRS-Tubitak no. 25784 grant,  from Universit\' e de Lille 1 and from  Labex CEMPI  (ANR-11-LABX-0007-01). The authors thank the anonymous referees for their valuable suggestions.}

\begin{abstract}
We  answer to a problem raised by recent work of Jelonek and Kurdyka: how can one detect by rational arcs the bifurcation locus of a polynomial map $\bR^n\to\bR^p$ in case $p>1$. We  describe an effective estimation of the ``nontrivial'' part of the bifurcation locus.

\end{abstract}

\maketitle

 
\section{Introduction}

The \textit{bifurcation locus} of a polynomial map $f\colon \bR^n \to \bR^p$,  $n \geq p$,  is the smallest subset $B(f) \subset \bR^p$ such that $f$ is a locally trivial $C^\ity$-fibration over $\bR^p\setminus B(f)$.  
It is well known that $B(f)$ is the union of  the set of critical values $f(\Sing f)$
  and the set of \textit{bifurcation values at infinity} $\Bi(f)$ (see Definition \ref{d:bif}) which may be non-empty and disjoint from $f(\Sing f)$ even in very simple examples. Finding the bifurcation locus in the cases $p>1$ or $p=1$ and $n>2$ is yet an unreached ideal. Nevertheless one can obtain approximations by supersets of $\Bi(f)$ from exploiting asymptotical regularity conditions  \cite{ST}, \cite{Pa}, \cite{Ra}, \cite{Ga-ity},   \cite{Ti-reg}, \cite{KOS}, \cite{HP}, \cite{DRT}, \cite{CT}, \cite{Je-banach}, \cite{NZ}, \cite{JT} etc.
 
Improving the effectivity of the detection of asymptotically non-regular values becomes an important issue, for instance it leads to applications in optimisation problems \cite{HP2}, \cite{Sa}.  Along this trend,
   Jelonek and Kurdyka \cite{JK} produced recently an algorithm for finding the set of \emph{asymptotically critical values} $\cK_\ity(f)$ in case $p=1$. It is known that in this case $\cK_\ity(f)$ is finite and includes $\Bi(f)$. 
   A sharper estimation of $\Bi(f)$ has been found in the real setting \cite{DT} by approximating the set of \emph{asymptotic $\rho_a$-nonregular values} of $f$. The later method provides a finite set of values $A(f)$ with the following property:   $\Bi(f)\subset  A(f) \subset  \cK_\ity(f)$.

In case $p>1$ the bifurcation locus $\Bi(f)$ may be no more finite. Actually, by the Morse-Sard result proved by Kurdyka, Orro and Simon \cite{KOS} for $\cK_\ity(f)$, or by the one  obtained in \cite{DRT} for the sharper estimation $\Bi(f) \subset \cS_0(f)\subset \cK_\ity(f)$, one only knows that the sets $\cK_\ity(f)$ and $\cS_0(f)$  are contained in a 1-codimensional semi-algebraic subsets of $\bR^p$. 

 Our approach is based on the set $\cS_\ity(f)$ of non-regular values at infinity with respect to the Euclidean distance function from any point as origin, and which includes $\Bi(f)$.   Since the set of critical values  $f(\Sing f)$ is the image of an algebraic set and the well-known estimation methods apply, we consider it as the ``trivial'' part of the job.  The most difficult task is to apprehend the complements of $f(\Sing f)$ to the bifurcation locus $\Bi(f)$.

We shall detect here the ``nontrivial'' part $N\Si(f)$ of the bifurcation locus at infinity (defined at \S \ref{ss:nontriv}) which,  roughly speaking, contains the values of $\cS_\ity(f)$ which are not comming from the branches at infinity of the singular locus $\Sing f$. 

This note answers a question raised by the results \cite{JK} and \cite{DT}, 
as of how can one detect the bifurcation locus by rational arcs in the case $p>1$.


More precisely, given a polynomial map 
$f=(f_1,\ldots,f_p) \colon\bR^n \to \bR^p$, $\deg f_i\leq d$, we find all the values of the ``nontrivial'' part $N\Si(f)$ of $\cS_\ity(f)$ and hence of nontrivial part $N\Bi(f)$ of the bifurcation locus $\Bi(f)$, as follows:\\

\noindent
(1). We consider a set of rational paths: $(x(t),y(t))=\left(\sum_{-ds \le i \le s} a_it^i,
\sum_{- ds \le j \le 0}b_j t^j\right)\subset \bR^n \times \bR^p$,  where   
$s =  [p(d-1)+1]^{n-p} [p(d-1)(n-p) +2]^{p-1}$.

This means a finite number of vectorial coefficients $a_i \in \bR^n$, for  $-ds \le i \le s$,   and $b_j \in \bR^p$, for $- ds \le j \le 0$. 
\smallskip

\noindent
(2).  The coefficients are subject to several conditions, namely:  $\| b_0\| = 1$,  $\exists k>0$, $a_k \not= 0\in \bR^n$, we ask the annulation of the coefficients of the terms with positive exponents in the expansion of $f(x(t))$ and  the annulation of the coefficients of 
 the terms with non-negative exponents in the expressions $x_i(t)\phi_{j}(x(t),y(t))$, for all $i,j\in\{1,\ldots,n\}$ (cf \eqref{eq:phi} for the definition).
 
 \medskip

We denote by $\Arc_{\ity}(f)$ the algebraic subset of arcs obtained by this construction (steps (1) and (2) above), and by  $\alpha_0(\Arc_{\ity}(f))$ the set of limits $\lim_{t\to\ity}f(x(t))$, i.e. the free coefficient in the expansion of $f(x(t)$ for $(x(t),y(t))\in \Arc_{\ity}(f)$.
 Then our main result, Theorem \ref{t:main2}, proves the inclusions:  
\[N\Si(f) \subset \alpha_0(\Arc_{\ity}(f))\subset \cK_{\ity}(f).\]


\section{Regularity conditions at infinity and bifurcation loci}\label{s:reg-ity}

\subsection{Bifurcation locus}
 Let $f=(f_1,\ldots,f_p)\colon\bR^n\to\bR^p$ be a polynomial map, $n\geq p$.

\begin{definition}\label{d:bif}
We say that $t_0\in\bR^p$ is a \textit{typical value}  of $f$ if there exists a disk $D\subset \bR^p$ centered at $t_0$ such that the restriction $f_{|}\colon f^{-1}(D) \to D$ is a locally trivial $C^\ity$-fibration. Otherwise we say that $t_0$ is a \textit{bifurcation value} (or atypical value). We denote by $B(f)$ the set of bifurcation values of $f$.

We say that $f$ is {\em topologically trivial at infinity at $t_0\in \bR^p$} if there exists
a compact set $\mathcal{K}\subset \bR^n$ and a disk $D \subset \bR^p$ centered at $t_0$ such that the restriction $f_| :  f^{-1}(D)\setminus \mathcal{K} \to D$  is a locally trivial $C^\ity$-fibration. 
Otherwise we say that $t_0$ is a {\em bifurcation value  at infinity of} $f$. We denote by $\Bi(f)$ the  bifurcation locus at infinity of $f$.
\end{definition}

\subsection{The rho-regularity}\label{ss:rho}
 Let $a=(a_1,\ldots,a_n)\in\bR^n$ and let  $\rho_a\colon\bR^n \to \bR_{\geq 0}$,   $\rho_{a}(x)=(x_1-a_1)^2+\ldots+(x_n-a_n)^2$, be the Euclidian distance function to $a$. 
Let $f\colon\bR^n\to\bR^p$ be a polynomial map,  where $n\geq p$. 

\begin{definition}[Milnor set at infinity and the $\rho_a$-nonregularity locus] \cite{DT} \\
\label{d:milnor}
 The critical set $\mathcal{M}_a(f)$ of the map $(f, \rho_a)\colon\bR^n\to\bR^{p+1}$ is called the \emph{Milnor set of $f$} (with respect to the distance function). The following semi-algebraic set, cf \cite[Theorem 5.7]{DRT} and \cite[Theorem 2.5]{DT}:
\begin{equation}
 \cS_a(f):=\{t_0\in\bR^p\mid\exists \{ x_j\}_{j\in \bN}\subset \mathcal{M}_a(f), \lim_{j\to\infty}\| x_j\|=\infty\mbox{ and }\lim_{j\to\infty}f(x_j)=t_0\}
\end{equation}
will be called the set of \emph{asymptotic $\rho_a$-nonregular values}. If $t_0\notin \cS_a(f)$ we say that $t_0$ is \textit{$\rho_a$-regular at infinity}. Let  
$\Si (f):= \bigcap_{a\in \bR^n} \cS_{a}(f)$. 
\end{definition}


\begin{lemma}\label{l:semi-alg}  $\Si(f)$ is a semi-algebraic set.
\end{lemma}
\begin{proof} Let $f\colon\bR^n\to\bR^p$ be a polynomial mapping and let us consider the following semi-algebraic set:
\[\mathcal{W} :=\{(x,a)\in\bR^n\times \bR^n \mid x\in \mathcal{M}_a(f) \}.\]

By the definition of $\Si(f)$, we have:
\[\Si(f):=\{y\in\bR^p\, \mid \forall a\in\bR^n, \exists \{(x_k,a)\}\subset \mathcal{W} \mbox{ such that } f(x_k)\to y\},\]
which tells that $\Si(f)$ can be writen by using first-order formulas.  This means that $\Si(f)$ is a semi-algebraic set, see for instance \cite[pag.28-29]{Coste-book} and \cite[Prop. 2.2.4]{Bochnak}.    
\end{proof}

It has been proved in \cite{Ti-reg}, \cite{DRT}, \cite{DT} that one has the inclusion $\Bi(f)\subset \cS_a(f)$, for any $a\in \bR^n$, thus in particular:
\begin{equation}\label{eq:bif}
\Bi(f)\subset \Si (f).
\end{equation}

It was believed, cf \cite[Conjecture 2.11]{DT}, that \eqref{eq:bif} was an equality. We show here by an example that this is not the case, at least in the real setting.

\subsection{Example for $B_{\infty}(f) \not= \cS_{\infty}(f)$}\label{ss:example} 
We consider the two-variable real polynomial\footnote{We thank Y. Chen for suggesting us to test this example.} constructed in \cite{TZ},  $f\colon\bR^2\to\bR$, $f(x,y)=y(2x^2y^2-9xy+12)$.  We show that $\cS_{\infty}(f) = \{ 0\}$ and $B_{\infty}(f) = \emptyset$. 
 
It was already proved in \cite{TZ} that $f$ has no singular value, no bifurcation value and that $\cS_{0}(f) \subset \{ 0\}$. We shall prove here that this inclusion is an equality. Moreover, we prove  here that $\{ 0\} \subset \cS_{a}(f)$ for any center $a\in \bR^2$.

For any fixed $a=(a_1,a_2)\in\bR^2$, we have: 
\[
\cM_a(f)=\{(x,y)\in\bR^2\mid y^2(4xy-9)(y-a_2)=6(x-a_1)(xy -1)(xy-2).\}
\]
For $x=0$ we eventually get solutions of the above equation but which have no influence on the set $\cS_a(f)$. By removing these solutions from $\cM_a(f)$, we pursue with the resulting set, which we denote by $\cM'_a(f)$. Thus, assuming that $x\neq 0$ and multiply the equation by $x^3$, we obtain:
\begin{equation}\label{eq:a1-a2-eq1}
\cM'_a(f)=\{(x,y)\in\bR^2\mid x^2y^2(4xy-9)(xy-xa_2)=6x^3(x-a_1)(xy -1)(xy-2)\}.
\end{equation}
We show that we can find solutions $(x_k,y_k)_{k\in \bN}$ of the equality in \eqref{eq:a1-a2-eq1} such that $\|(x_k,y_k)\|\to\infty$ and  $f(x_k,y_k)\to 0$.  Indeed, 
setting $z:= xy$ our equation  \eqref{eq:a1-a2-eq1} becomes $z^2(4z-9)(z-a_2 x)=6x^3(x-a_1)(z -1)(z-2)$. We then consider each side as a curve of variable $z$ with $x$ as parameter.  We consider the graphs of these two curves 
and observe that for each sign of $a_{2}$ the two graphs intersect at least once for any fixed and large enough $|x|$
and that this happens at some value of $z$ in the interval $]0,1[$ (and  in the interval $]1,2[$ in case $a_{2}=0$, respectively). This shows that we can find solutions $(x_k,y_k)\in \cM_a(f)$ with modulus tending to infinity and, since $z_{k} = x_{k}y_{k}$ is bounded and $y_{k}$ tends to 0, we get that $f(x_k,y_k)\to 0$.

In conclusion, we have shown that $\cS_{\infty}(f) = \{ 0\}$, which implies $\Bi(f) \not= \cS_{\infty}(f)$.

\subsection{Generic dimension of the nonsingular part of the Milnor set}\ 

The following statement has been noticed in case $p=1$ in \cite{HP} (see also \cite[Lemma 2.2]{Dut} or \cite{DT}). We outline the  proof in case $p> 1$, some details of which will be used in \S \ref{s:detec-S}.

\begin{lemma}\label{l:gen-a} Let $f=(f_1,\ldots,f_p)\colon\bR^n\to\bR^p$ be a polynomial map, where $n>p$ and $\deg f_i\leq d, \forall i$. There exists an open dense subset $\Omega_f\subset \bR^n$ such that, for every $a\in \Omega_f$,  the set $\mathcal{M}_a(f)\setminus\Sing f$ is either  a smooth manifold of dimension $p$, or it is empty. 
\end{lemma}
\begin{proof} 
We denote by $M_{I}[\rD (f)(x)]$ (respectively $M_{I}[\rD (f,\rho_a)(x)]$) the minor of the Jacobian matrix $\rD (f)(x)$ (respectively $\rD (f,\rho_a)(x)$) indexed by the multi-index $I$. We set 
\begin{equation}
Z:=\{(x,a)\in\bR^n\times\bR^n\mid x\in \mathcal{M}_a(f)\setminus\Sing f\}.
\end{equation}
If $Z=\emptyset$, then  $\mathcal{M}_a(f)\setminus\Sing f = \emptyset, \forall a\in\bR^n$. From now on let us consider the case that $Z\neq\emptyset$. 
Let $(x_0,a_0)\in Z$. Since $\Sing f$ is closed, there is a neighborhood   $U\subset \bR^n$ of $x_0$ such that $U\cap\Sing f=\emptyset$. This means that there exists a multi-index $I=(i_1,\ldots,i_p)$ of size $p$, $1\leq i_1 < \ldots <i_p\leq n$, such that $M_I[\rD f(x)]\neq 0$, $\forall x\in U$. 

Let $S_I:=\{J=(j_1,\ldots,j_{p+1})\mid  I\subset J\}$ be the set of multi-indices of size $p+1$ such that $1\leq j_1 <\ldots < j_{p+1}\leq n$ and $i_1,\ldots,i_p\in \{j_1,\ldots,j_{p+1}\}$.  There are $(n-p)$ multi-indices $J\in S_I$; we set
\begin{equation}\label{eq:m-j}
m_J(x,a):=M_{J}[\rD (f,\rho_a)(x)], (x,a) \in U\times \bR^n. 
\end{equation} 
From the definitions of $Z, U$ and the functions $m_J$, we have: 
\begin{equation}\label{eq:m-j-1}
Z\cap(U\times \bR^n)=\{(x,a)\in U\times\bR^n\mid m_J(x,a)=0; \forall J\in S_I\}.
\end{equation}
Let  $\varphi\colon U\times \bR^n\to\bR^{n-p}$ be the map consisting of the functions $m_{J}$ for $J\in S_I$. Then $\varphi^{-1}(0)=Z\cap (U\times\bR^n)$ and we notice that $\rD\varphi(x,a)$ has rank $(n-p)$ at any  $(x,a)\in U\times\bR^n$. Indeed, let 
\[\left( \frac{\partial \varphi}{\partial a_k} (x,a) \right)_{(n-p)\times(n-p)}, \, k\notin I, (x,a)\in U\times \bR^n.\]
This is a minor of $\rD\varphi(x,a)$ of size $(n-p)$. Interchanging if necessary the order of its lines, it is a diagonal matrix with all the entries on the diagonal equal to  $- M_I[\rD f(x)] $ and hence non-zero. This and \eqref{eq:m-j-1} show that $Z$ is a manifold of dimension $n+p$. 

We next consider the projection $\tau\colon Z\to\bR^n$, $\tau(x,a)=a$. Thus, $\tau^{-1}(a)=(\mathcal{M}_a(f)\setminus\Sing f)\times\{a\}$. By Sard's Theorem, we conclude that,  for almost all $a\in\bR^n$, $\tau^{-1}(a) = (\mathcal{M}_a(f)\setminus\Sing f)\times\{a\}\cong (\mathcal{M}_a(f)\setminus\Sing f)$ is either a smooth manifold of dimension $p$ or  an empty set. 
\end{proof}

\subsection{The relation to the Malgrange-Rabier condition}\label{ss:mal-ga-ra}

\begin{definition}[\cite{Ra}]\label{d:ra-ity}
Let $f\colon\bR^n\to\bR^p$ be a polynomial map, $n\geq p$. Denote by $\rD f(x)$ the Jacobian matrix of $f$ at $x$. We consider
\begin{eqnarray}\label{eq:ra-ity} \cK_{\ity}(f) & := & \{t\in\bR^{p}\mid \exists \{ x_{j}\}_{j\in \bN} \subset \bR^{n}, \lim_{j\to\ity}\|
x_j\|=\ity, \\ \nonumber
 &  &  \lim_{j\to\ity}f(x_j)= t \mathrm{\ and\ }\lim_{j\to\ity}\|x_j\|\nu(\rD f (x_j))=0\},
\end{eqnarray}
where 
\begin{equation}\label{eq:ra-func}
\nu(A):=\inf_{\|y\|=1}\|A^*(y)\|,
\end{equation} 
for a linear map $A$ and its adjoint $A^*$. 

We call the set $\cK_{\ity}(f)$ of \emph{asymptotic critical values of $f$}. If $t_0\notin \cK_{\ity}(f)$ we say that $f$ verifies the Malgrange-Rabier condition at $t_0$.
\end{definition}


We have the following relation between $\rho_a$-regularity and Malgrange-Rabier condition:

\begin{theorem}[{\cite[Th. 2.8]{DT}}]\label{p:rho-rab}
 Let $f=(f_1,\ldots,f_p)\colon\bR^n\to\bR^p$ be a polynomial  map, where $n> p$. Let $\phi : ]0,\varepsilon [ \to \mathcal{M}_a(f)\subset \bR^n$ be an analytic path such that  $\lim_{t\to 0}\|\phi(t)\|=\infty$ and $\lim_{t\to 0}f(\phi(t))= \mathrm{c}$. Then $\lim_{t\to 0}\|\phi(t)\| \nu(\rD f (\phi(t)))=0$.
In particular
 $\cS_a(f)\subset \cK_{\infty}(f)$ for any $a\in\bR^n$, and $\Si (f)\subset \cK_{\infty}(f)$.
\fin
\end{theorem}

\begin{remark}\label{r:1}
 See \cite{DRT} and more precisely \cite[Theorem 2.5]{DT} for a structure result and a fibration result on $\Si (f)$. The inclusion $\Si (f)\subset \cK_{\infty}(f)$ may be strict (e.g. \cite{PZ} and \cite[Example 2.9]{DT}).
 The inclusion $B_{\infty}(f)\subset  \cS_{\infty}(f)$ may be strict, see the above Example \S \ref{ss:example}.
 One may also have $\cS_a(f)\neq \cS_b(f)$ for some $a\neq b$, see  \cite[Example 2.10]{DT}.
\end{remark}

\subsection{The nontrivial bifurcation locus at infinity}  \label{ss:nontriv}
 We have discussed up to now three types of bifurcation loci: $\Bi(f)$, $\Si(f)$ and $\cK_{\ity}(f)$.
 All of them may contain points of the critical locus $f(\Sing f)$. This locus can be estimated separately since it is the image by $f$ of an algebraic set and the known estimation methods apply. What is more difficult to apprehend are the respective complements of $f(\Sing f)$. We define here the ``nontrivial parts''
of the bifurcation loci and next describe a procedure to estimate the one of $\Si(f)$.

From the definitions of $\mathcal{M}_a(f)$ and $\cS_a(f)$, we have the equality $\cS_a(f) = J(f_{|\mathcal{M}_a(f)})$, where  $J(f_{|\mathcal{M}_a(f)})$ is the \emph{non-properness set} of $f_{|\mathcal{M}_a(f)}$. Jelonek defined this set in general: 

  \begin{definition}\label{d:jelonek}(\cite[Definition 3.3]{Je-mathann}, \cite{JK}).
 Let $g : M \to N$ be a continuous map, where $M,N$
are topological spaces. One says  that $g$ is \emph{proper at the value $t\in N$} if there exists an open neighbourhood
$U\subset N$ of $t$ such that the restriction $g_{|g^{-1}(U)} : g^{-1}(U)\to U$ is a proper map. We denote by $J(g)$ the set of points at which $g$ is not proper.
\end{definition}

\medskip

In our setting $f\colon\bR^n\to\bR^p$,  let us define the \emph{nontrivial $\rho$-bifurcation set at infinity}
$N\Si(f) := \bigcap_{a\in \bR^n} N\cS_a(f)$,
where: 
\[ N\cS_a(f) := \{t\in\bR^{p}   \mid \exists \{ x_{j}\}_{j\in \bN} \subset \mathcal{M}_a(f)\setminus \Sing f,  \lim_{j\to\ity}\|
x_j\|=\ity,  \nonumber
\mathrm{\ and\ }  \lim_{j\to\ity}f(x_j)= t  \}\]
and note that
$\Si(f)  = N\Si(f) \cup J(f_{|\Sing f})$
and that $N\Si(f)$ is a \emph{closed set} since  each set $N\cS_a(f)$ 
is closed, which fact follows from the arguments of \cite[Theorem 5.7(a)]{DRT}.

Similarly, we introduce the following notation for the \emph{nontrivial bifurcation set at infinity} which is the object of our main result, Theorem \ref{t:main2}:

\begin{eqnarray}\label{eq:nontrivB} 
N\Bi(f) :=  \Bi(f)\setminus J(f_{|\Sing f}).
\end{eqnarray}

By the above definitions and by Theorem \ref{p:rho-rab}, we immediately get:

\begin{proposition}\label{p:inclusion}
 \[ N\Bi(f) \subset N\Si(f)\subset \cK_{\ity}(f).
  \]
 \fin
\end{proposition}

\begin{remark} If $f$ has a compact singular set $\Sing f$ or, more generally, if $J(f_{|\Sing f}) = \emptyset$, then $N\Si(f) = \Si(f)$,  and $N\Bi(f) = \Bi(f)$.  However these equalities mai fail whenever $J(f_{|\Sing f}) \not= \emptyset$.

In this matter, let us point out here that the proofs of \cite[Proposition 3.1, Theorem 3.4]{DT}  run actually for the set $N\Si(f)$; one therefore needs to replace $\Si(f)$ by $N\Si(f)$ in the statements of those results. 
 
\end{remark}

\medskip

\section{Detection of bifurcation values at infinity by parametrized curves}\label{s:detec-S}



\subsection{Effective Curve Selection Lemma at infinity via the Milnor set} \ 

If $t_0 \in N\Si(f)$ then $t_0 \in N\cS_a(f)$ for any $a\in \bR^n$ and in particular for  $a\in \Omega_f$, where $\Omega_f$ is as in Lemma \ref{l:gen-a}. 

\begin{theorem}\label{l:main4}
Let $f=(f_1,\ldots,f_p)\colon\bR^n\to\bR^p$ be a polynomial mapping such that $\deg f_i\leq d, \forall i = 1, \ldots, p$, and $n>p$. Let $t_0 \in N\cS_a(f)$ for some $a\in \Omega_f$. Then there exists  an analytic path: 
\begin{equation}\label{eq:curveselection}
x(t)=\sum_{-\ity\leq i\leq s} a_it^i,
\end{equation}
with 
\[ s \le  [p(d-1)+1]^{n-p} [p(d-1)(n-p) +2]^{p-1}\]
 and such that: 
\begin{enumerate}
{\rm\item $x(t)\in\mathcal{M}_a(f)\setminus \Sing f$,   for  any $t\ge R$, for some large enough $R\in \bR_{+}$;
\item $\|x(t)\|\to \infty,$} as $t\to\infty ;$
{\rm\item $f(x(t))\to t_0,$} as $t\to\infty$.
\end{enumerate}
\end{theorem}

\begin{proof} 
 The case $p=1$ is \cite[Theorem 3.4]{DT}. We assume in the following that $p>1$. 

From Lemma \ref{l:gen-a} we have  that $\mathcal{M}_a(f)\setminus\Sing f$  is a smooth semi-algebraic set of dimension $p$ since non-empty by our hypothesis on $t_{0}$.  From the proof of Lemma \ref{l:gen-a} the set 
 $\mathcal{M}_a(f)\setminus \Sing f$ is locally a complete intersection defined by $(n-p)$ equations, each of which is of degree at most $p(d-1)+1$.  So let us denote by $g_{1}, \ldots , g_{n-p}$ these functions.

We use coordinates $(x_1,\ldots,x_n)$ for the affine space $\bR^n$ and coordinates $[x_0:x_1:\ldots:x_n]$ for the projective space $\bP^n$. We identify the affine space $\bR^n$ with the chart $\{x_0\neq 0\}$ of $\bP^n$. Let $\bX=\overline{\graph f}$ be the closure of the graph of $f$ in $\bP^n\times \bR^p$ and let $\bX^\ity$ the intersection of $\bX$ with the hyperplane at infinity $\{ x_0=0\}$. Let  $i : \bR^n \to \bX\subset \bP^n \times \bR^p$, $x\mapsto (x, f(x))$ be the graph embedding. Consider the closure in $\bX$ of the image $i(\mathcal{M}_a(f)\setminus \Sing f)$ and denote it (abusively) by $\overline{\mathcal{M}_a(f)\setminus \Sing f}$.

Let  then $w:=(\underline{x},t_0)\in \overline{\mathcal{M}_a(f)\setminus \Sing f} \cap \bX^\ity$. We shall work in some 
affine chart $U \simeq \bR^n$ at infinity of $\bP^n$ assuming (without loss of generality) that the point $\underline{x}$ is the origin.  We may then use an ``effective curve selection lemma'' to show that there is a curve $\Gamma \subset \mathcal{M}_a(f)\setminus \Sing f$ such that $w\in \overline{\Gamma}$  and that this curve has a one-sided bounded parametrization. To do so, we combine Milnor's basic construction in \cite{Mi} with the idea of Jelonek and Kurdyka given in \cite[Lemma 6.4]{JK}. 

  Namely we consider small enough spheres centered at $w\in U$ of equation $\rho_w = \beta$ and a function $h_{l} := x_{0}l$, for some linear function $l$ in the local coordinates. One can then prove like in \cite[Lemma 6.4]{JK} (where an apparently more particular situation was considered, but the proof works as well) that, for a general such linear function $l$, the set of critical points of the map $(\rho_{w}, h_{l}) : U \cap \overline{\mathcal{M}_a(f)\setminus \Sing f}\to \bR_{+}\times \bR$
 is an analytic  curve and its branches are the singular points of the restrictions of the quadratic function $h_{l}$ to the levels $\{\rho_{w} = \beta \} \cap \mathcal{M}_a(f)\setminus \Sing f$. It is shown in \cite[Lemmas 6.5 and 6.6]{JK} that these singular points are all Morse for a generic choice of $l$, and that there is at least one Morse point on each level, for small enough $\beta >0$.

 Let us then consider a branch of this analytic curve as our $x(t)$.  By its definition, this curve is a solution of the following  system of equations: $g_1=0, \ldots , g_{n-p} =0$ and $\d g_1 \wedge  \cdots  \wedge \d g_{n-p} \wedge \d \rho_w \wedge \d h_l =0$,
 the first of which are of degree at most $p(d-1)+1$ and the last one means the annulation of $p-1$ minors of degree  at most
 $p(d-1)(n-p)+2$.  Thus our algebraic set of solutions has degree $\delta$ verifying the inequality:
  \[  \delta \le [p(d-1)+1]^{n-p} [p(d-1)(n-p) +2]^{p-1}. \]

 Finally, by using the effective Curve Selection Lemma of Jelonek and Kurdyka \cite[Lemma 3.1 and Lemma 3.2]{JK} which says that there exists a parametrization of our curve $x(t)$ bounded by the degree $\delta$ of the curve,  we get exactly an expansion like \eqref{eq:curveselection}.   
This finishes the proof of our theorem. 
\end{proof}
 
\subsection{Finite length expansion for curves detecting asymptotically critical values}\
 We need a preliminary result which follows by applying \cite[Lemma 3.3]{JK} to each function $h_i$ in the following statement:
 
\begin{lemma}\label{l:main2}
 Let $h=  (h_1,\ldots,h_m) \colon\bR^k\to\bR^m$  be a polynomial map  and $\deg h_i\leq \tilde{d}, \forall i.$ Let $x(t)=\sum_{-\infty\leq i\leq s}a_it^i,$ where $t\in \bR$ , $a_i\in\bR^k$, $s > 0$ and that $\|x(t)\|\to\infty$ and $h(x(t))\to b.$ Then, for any $D\leq -\tilde{d}s+ s$, the truncated curve 
\[\tilde x(t)=\sum_{D\leq i\leq s}a_it^i,\]
verifies $\|\tilde x(t)\|\to\infty$ and $h(\tilde x(t))\to b$.
\fin
\end{lemma}

If we try to replace $x(t)$ given in \eqref{eq:curveselection} by a truncated path, we may go out of the set $\mathcal{M}_a(f)\setminus \Sing f$. Bearing in mind the inclusion  $\cS_a(f)\subset \cK_{\infty}(f)$ of Theorem \ref{p:rho-rab}, instead of searching in vain a truncated expansion inside the Milnor set,  we may show that there exists a truncation which verifies the Malgrange-Rabier condition
\eqref{eq:ra-ity}.
The  proof of the following result  employs the technique of \cite[Theorem 3.2]{DRT} and \cite[Theorem 2.4.8]{Di}, where we have used the $t$-regularity to find a geometric interpretation for $\cK_{\infty}(f)$.

\begin{proposition}\label{p:main3} 
Let $f=(f_1,\ldots,f_p)\colon\bR^n\to\bR^p$ be a polynomial map such that $n>p$ and that $\deg f_i\leq d, \forall i.$ Let 
\[ x(t)=(x_1(t),\ldots,x_n(t))=\sum_{-\infty\leq i\leq s}a_it^i,\]
where $t\in\bR$, $a_i\in\bR^n$, $s > 0$ and such that: 
\begin{enumerate}
{\rm\item $\|x(t)\|\to\infty,$} as $t\to\infty$;
{\rm \item $f(x(t))\to b,$} as $t\to\infty$;
{\rm\item $\|x(t)\| \nu(\rD f (x(t)))\to 0,$} as $t\to \infty$.
\end{enumerate} 
Then the truncated expansion  
\[\tilde x(t)=\sum_{-ds\leq i\leq s}a_it^i,\]
  verifies the following conditions:
 \begin{itemize}
{\rm\item[(i)] $\|\tilde x(t)\|\to\infty,$} as $t\to\infty$;
{\rm\item[(ii)] $f(\tilde x(t))\to b,$}  as $t\to\infty$;
{\rm\item[(iii)] $\|\tilde x(t)\| \nu(\rD f (\tilde x(t)))\to 0,$} as $t\to \infty$.
\end{itemize}

\end{proposition}

\begin{proof}
We treat here the case $p>1$. See Remark \ref{re:pneq1} for the case $p=1$.

By the definition of $\nu$ (Definition \ref{d:ra-ity} and \eqref{eq:ra-func}),  condition (c) means:

\begin{equation}\label{eq:nu-1}
  \|x(t)\|\, \left( \inf_{\|y\|=1}\|\rD f (x(t))^*(y)\| \right) \to 0, \mbox{ as } t \to \infty,
\end{equation} 
where $\rD f (x(t))^*$ denotes the adjoint  of $\rD f (x(t))$. 

 Since $\nu$ is a semi-algebraic mapping (see e.g \cite[Proposition 2.4]{KOS}), the Curve Selection Lemma and \eqref{eq:nu-1} imply that  there there exists an analytic path (see also the proofs of \cite[Theorem 3.2]{DRT} and \cite[Proposition 2.4]{CDTT} for this argument):

\[y(t)=\sum_{-\infty\leq i\leq 0}b_jt^j=(y_1(t),\ldots,y_p(t)),b_j\in\bR^p,\]
 such that $\|y(t)\|=1,\forall t \gg 0$, and that:
\begin{equation} \label{eq:ra-ga1}
 \|x(t)\|\left\|y_1(t)\frac{\partial f_1}{\partial x}(x(t))+\cdots+y_p(t)\frac{\partial f_p}{\partial x}(x(t))\right\|\to 0, \mbox{ as } t\to\infty,
\end{equation}
where $\frac{\partial f_i}{\partial x}(x(t)) := \left(\frac{\partial f_i}{\partial x_1}(x(t)),\ldots,\frac{\partial f_i}{\partial x_n}(x(t))\right)$ for $i=1,\ldots,p$.

For any fixed $j\in  \{1,\cdots,n\}$ we set $\phi_{j}\colon\bR^n\times\bR^p\to\bR$,
\begin{equation}\label{eq:phi}
 \phi_{j}(x,y) := \left(y_1 \frac{\partial f_1}{\partial x_j}(x)+\cdots+y_p \frac{\partial f_p}{\partial x_j}(x)\right).
\end{equation}
 
 It then follows that $\deg  \phi_{j}\leq d$ and that our path: 
\[ (x(t),y(t)) :=\left(\sum_{-\infty\leq i\leq s}a_it^i, \sum_{-\infty\leq i\leq 0}b_jt^j\right) \]
verifies the conditions:
\begin{enumerate}
\item[(1)] $\|x(t)\|\to\infty$  as $t\to\infty$,  and $\|y(t)\| = 1$;
\item[(2)] $x_i(t)\phi_{j}(x(t),y(t))\to 0$ as $t\to\infty,$ for any $i,j\in\{1,\ldots,n\}.$
\end{enumerate}

Applying Lemma \ref{l:main2} to the mapping $(x_i\phi_{j})_{i,j=1}^n$, we get that, for any $D\leq -(d+1)s+s = -ds$,  the truncated path: 

\[(\tilde x(t),\tilde y(t)) :=\left(\sum_{D\leq i\leq s}a_it^i, \sum_{D\leq i\leq 0}b_jt^j\right)\]
verifies the conditions:
\begin{enumerate}[label={(\alph*')}]
\rm\item[(1')]   $\|\tilde x(t)\|\to\infty$ and $\|\tilde y(t)\|\to 1$ as $t\to\infty;$ 
\rm{\item[(2')] $\tilde x_i(t)\phi_{j}(\tilde x(t),\tilde y(t))\to 0,$} as $t\to\infty,$ for any $i,j\in\{1,2,\ldots,n\}.$
\end{enumerate}
These imply: 
\begin{equation}\label{eq:ra-ga3}
 \|\tilde x(t)\|\left\|\tilde y_1(t)\frac{\partial f_1}{\partial x}(\tilde x(t))+\cdots+\tilde y_p(t)\frac{\partial f_p}{\partial x}(\tilde x(t))\right\|\to 0 \mbox{ as } t\to\infty,
\end{equation}
and, since $\|\tilde y(t)\|\to 1$, we obtain:
\begin{equation}\label{eq:ra-ga4}
 \|\tilde x(t)\|\frac{1}{\|\tilde y(t)\|}\left\|\tilde y_1(t)\frac{\partial f_1}{\partial x}(\tilde x(t))+\cdots+\tilde y_p(t)\frac{\partial f_p}{\partial x}(\tilde x(t))\right\|\to 0, \mbox{ as } t\to\infty.
\end{equation}
The later  implies that $\|\tilde x(t)\| \nu(\rD f (\tilde x(t)))\to 0,$ as $t\to \infty$, which shows (iii).

Next,  (i) follows by  (1'), and (ii) follows from Lemma \ref{l:main2} for $h:= f$,  since $-ds<-ds+s$. 

\end{proof}

\begin{remark}\label{re:pneq1}
In case $p=1$, in the proof of Proposition \ref{p:main3} we may consider $\phi_{j}\colon\bR^n \to\bR,\phi_{j}(x,y)=  \frac{\partial f}{\partial x_j}(x)$ since in this case $y=1$.  
Then $\deg  \phi_{j}\leq d-1 $ and by applying Lemma \ref{l:main2} as above to the mapping $(x_i\phi_{j})_{i,j=1}^n$  we get that, for any $D\leq -ds+s$, the truncation $\tilde{\tilde{x}}(t)= \sum_{D\leq i\leq s}a_it^i$ satisfies (i), (ii) and (iii). 

In the definition  of $\Arc(f)$, the lower bound is $-ds+s$ instead of $-ds$. Since the value of the degree $s$ from Theorem \ref{l:main4} is $d^{n-1}$ in case  $p=1$, we recover the result in \cite{DT}.  
\end{remark}

\subsection{Arc space and the main result}

We may now apply to a polynomial map $f=(f_1,\ldots,f_p) \colon\bR^n \to \bR^p$, $\deg f_i\leq d$, a similar procedure as the one described by Jelonek and Kurdyka \cite{JK} in case $p=1$. 
Thus, in case $p>1$, we consider  the following space of arcs associated to 
$f$: 
\begin{equation}\label{eq:arc} 
\Arc(f) :=
\left\{ (x(t),y(t))=\left(\sum_{-ds \le i \le s} a_it^i,
\sum_{- ds \le j \le 0}b_j t^j\right),  \ (a_i,b_i)\in\bR^n\times\bR^p 
\right\},
\end{equation}

\noindent
where $s:= [p(d-1)+1]^{n-p} [p(d-1)(n-p) +2]^{p-1}$,  as in Theorem \ref{l:main4}. Then $\Arc(f)$ is a vector space of finite dimension. 

Referring to the notations in \eqref{eq:arc}, we define, in a similar manner as \cite[Definition 6.10]{JK}, the  \emph{asymptotic variety of arcs} $\Arc_{\ity}(f)\subset \Arc(f)$, as the algebraic subset  of the rational arcs $(x(t),y(t))\in \Arc(f)$ verifying the following conditions: 

\medskip
\begin{enumerate}
\item[(a')] $\exists k>0$ such that $a_k \not= 0\in \bR^n$, and  $\| b_0\| = 1$.
\item[(b')] $\ord_t f(x (t))\le 0$.
\item[(c')] $\ord_t \left( x_i(t)\phi_{j}(x(t),y(t))\right)<0,$ for any $i,j\in\{1,\ldots,n\},$ where $\phi_{j}$ is defined at \eqref{eq:phi} in the proof of Proposition \ref{p:main3}.
\end{enumerate}

\medskip
\noindent
Let us then set $\alpha_0 \colon \Arc_{\ity}(f) \to \bR^p$, $\alpha_0(\xi(t)) := \lim_{t\to\ity}f(x(t))$, where $\xi(t)=(x(t),y(t))$.

\smallskip

In view of the above results, we may now give an estimation of the nontrivial $\rho$-bifurcation set at infinity $N\Si(f)$, thus of the nontrivial bifurcation locus $N\Bi(f)$, cf Proposition \ref{p:inclusion}:

\begin{theorem}\label{t:main2} $N\Si(f) \subset \alpha_0(\Arc_{\ity}(f))\subset \cK_{\ity}(f)$.
\end{theorem} 

\begin{proof}
If $\alpha\in N\Si(f)$ then $\alpha\in N\cS_a(f)$ for any fixed $a\in \Omega_f$. By Theorem \ref{l:main4},  there exists a path 
 \[x(t)= \sum_{-\infty\leq i\leq s}a_it^i \in \mathcal{M}_a(f)\setminus \Sing f,\]
 such that $\lim_{t \to \infty} f(x(t))= \alpha$.
 It follows from Theorem \ref{p:rho-rab} that $x(t)$ verifies the conditions (a)--(c) of Proposition \ref{p:main3}. 
Moreover, the truncation $\tilde x$ defined in the same Proposition \ref{p:main3}  verifies the properties (i)--(iii). Since conditions (i)--(iii) are equivalent to conditions (a')--(c'), we conclude that the first inclusion holds.  

 The second inclusion  $\alpha_0(\Arc_{\ity}(f))\subset \cK_{\ity}(f)$ is a direct consequence of the definitions of $\Arc_{\ity}(f)$ and $\cK_{\infty}(f)$ since properties (a'), (b') and (c') characterize the values $\alpha_0\in \cK_{\infty}(f)$ as shown in the proof of Proposition \ref{p:main3}.  This completes our proof.
\end{proof}

Let us remark that the first inclusion can be strict, as shown by the next example:

\begin{example}[{\cite[Example 2.10]{DT}}] Let $f\colon\bR^2\to \bR$, $f(x,y)=y(x^2y^2+3xy+3)$. We have $N\Si(f)=\emptyset$, $0\in \alpha_0(\Arc_{\ity}(f))$ and $0\in \cK_{\ity}(f)$.
\end{example}


  In trying to prove the equality in place of the second inclusion in Theorem \ref{t:main2} one notices that the inverse inclusion depends on the possibility of truncating paths which detect some value $\alpha_0\in \cK_{\ity}(f)$ at the order provided by Theorem 
\ref{l:main4}. But our  Theorem 
\ref{l:main4} is based on paths in the Milnor set $\mathcal{M}_a(f)\setminus \Sing f$, which provide in principle lower degrees than working with the Malgrange-Rabier condition  (\ref{eq:ra-ity}), and we know that the later is not equivalent to $\rho$-regularity (cf \S \ref{s:reg-ity}).    Else, for the same reason, it would be difficult to obtain examples to disprove the inverse inclusion.



\end{document}